\documentclass[11pt,a4paper]{article}
\usepackage[english]{babel}
\usepackage[utf8]{inputenc}
\usepackage[T1]{fontenc}
\usepackage{amsmath, amssymb, amsthm}
\usepackage{graphicx}
\usepackage{hyperref}
\usepackage{geometry}
\usepackage{float}
\usepackage[english]{datetime2}
\usepackage{authblk} 
\usepackage{booktabs}
\usepackage{subcaption} 
\usepackage{textcomp}

\DeclareUnicodeCharacter{02BF}{\textquoteleft} 
\newtheorem{proposition}{Proposition}
\title{The Angular Seed Power Map: A Constructive Approach to Recursive Scaling Spirals}

\author{Arjen T. Dijksman, PhD\thanks{Independent Researcher. Former affiliation: ESPCI Paris. Contact: materion@gmail.com}}
\DTMsavedate{origdate}{2026-06-23} 
\date{\DTMusedate{origdate}}

\begin{document}

\maketitle

\begin{abstract}
We present the ``Power Spiral Map'', a continuous angular evolution of the linear coordinate grid established in our previous work. While that previous Power Map utilized a seed value translating along a horizontal axis, this work builds upon a seed angle ($\theta$) projected onto a unit diameter circle. This operation controls two coupled geometric behaviors: an internal area-preserving partition of unity within a reference square ($\cos^2\theta + \sin^2\theta = 1$) and an external recursive scaling mechanism ($\sec\theta$ and $\cos\theta$) that dictates the expansion or contraction of successive generations of squares unfolding as a spiral in the 2D plane. We demonstrate that continuous variation of this angular parameter generates discrete geometric alignments that yield polynomial identities, with examples of the Golden Ratio ($\phi$) and the Plastic Ratio ($\psi$) defined through purely planar intersections.
\end{abstract}

\section{Introduction}
In our previous paper, \textit{The Power Map: A Unified Geometric Locus for Rational Exponents}, we formalized a planar coordinate framework where higher-order algebraic power curves were mapped as a continuous field of linear coordinate intersections $(s^p, s^q)$~\cite{dijksman2026powermap}. This structure depended on a linear seed value $s$ sliding along the horizontal axis, projecting rays through a fixed unit line at $x=1$ to construct a parametric lattice.

This paper builds directly upon that dimensional-homogeneity framework by shifting from a linear translation to a continuous polar rotation. Instead of a sliding scalar seed $s$, we introduce a continuous \textit{angular parameter} $\theta$ operating within a fixed unit diameter circle.

By allowing the boundaries of our foundational geometry to vary continuously along orthogonal constraints, the static coordinate grid unfolds into a family of recursive geometric spirals. This transition from a linear to a rotational seed uncovers an intrinsic link between localized trigonometric partitions of unity and global geometric progressions, establishing a continuous planar framework where specific intersection configurations correspond directly to the roots of non-linear polynomial systems.

\section{The Unit Diameter Circle and its Associated Square}

The foundation of the Angular Seed Power Map is a unit diameter circle and its associated unit square. We present this core geometry into three progressive stages, where a singular angular parameter $\theta$ translates local trigonometric partitions into a recursive scaling spiral.

\subsection{The Classical Pythagorean Partition and the Square Root Scale}

\textbf{Construction.} 
The configuration begins with a unit diameter circle $C_u$ anchored at $O(0,0)$ and $I(1,0)$. Atop this diameter, we construct the unit reference square $OIAB$ with top vertices $A(1,1)$ and $B(0,1)$, serving as the fundamental unit reference from which the entire construction is derived. We introduce a sliding point $S$ on the circle. As illustrated in Figure~\ref{fig:unit_cell_a}, $S$ defines a ray projected from the origin $O$ at a seed angle $\theta$. By Thales’ Theorem, the inscribed angle $\angle OSI$ is rigidly constrained to a right angle ($\angle OSI = \frac{\pi}{2}$). 

An orthogonal projection dropped from $S$ to the base diameter $OI$ partitions the unit length into two linear segments: a left segment $OC$ and a right segment $CI$. Given that the height of the reference square is $1$, a vertical line dropped through $C$ hits the upper edge $BA$ at point $D$. This line divides the unit square $OIAB$ into two distinct rectangles, $OCDB$ and $CIAD$. For geometric comparison, we erect standard squares $OSJK$ and $SIPQ$ directly upon the two legs of the right triangle $\triangle OSI$.

\begin{figure}[t!]
    \centering
    \begin{subfigure}{0.43\textwidth}
        \centering
        \includegraphics[width=\textwidth]{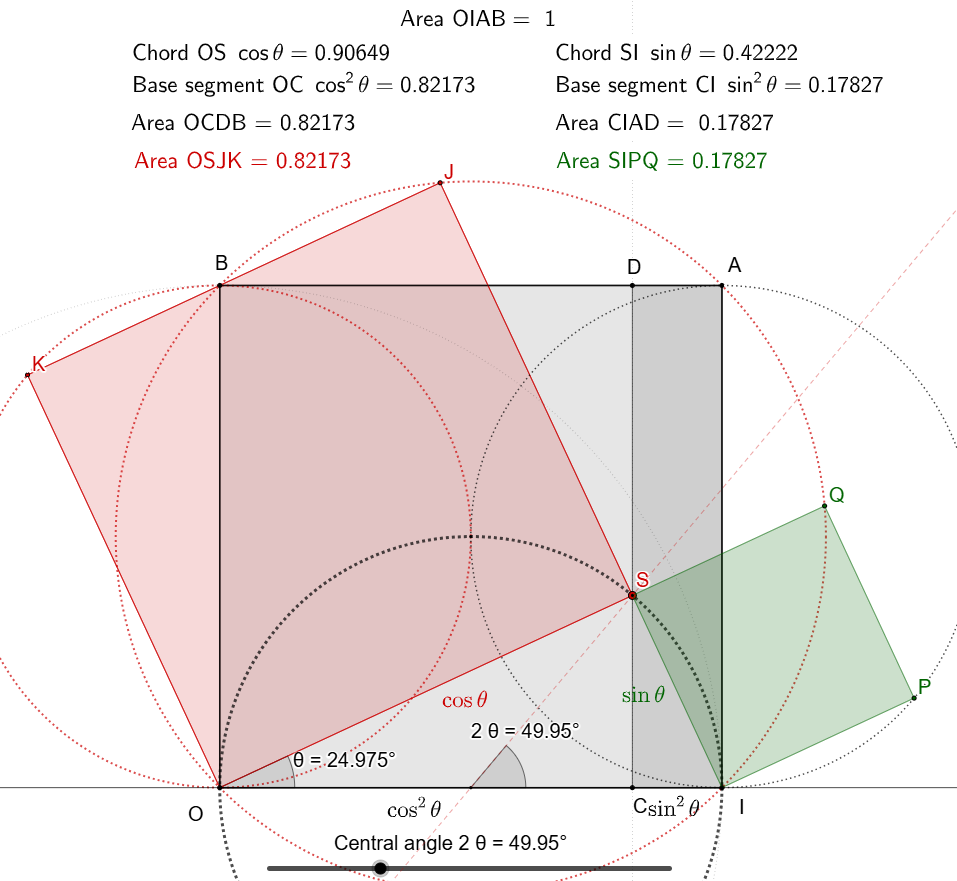}
        \caption{The Pythagorean partition.}
        \label{fig:unit_cell_a}
    \end{subfigure}
    \hfill
    \begin{subfigure}{0.52\textwidth}
        \centering
        \includegraphics[width=\textwidth]{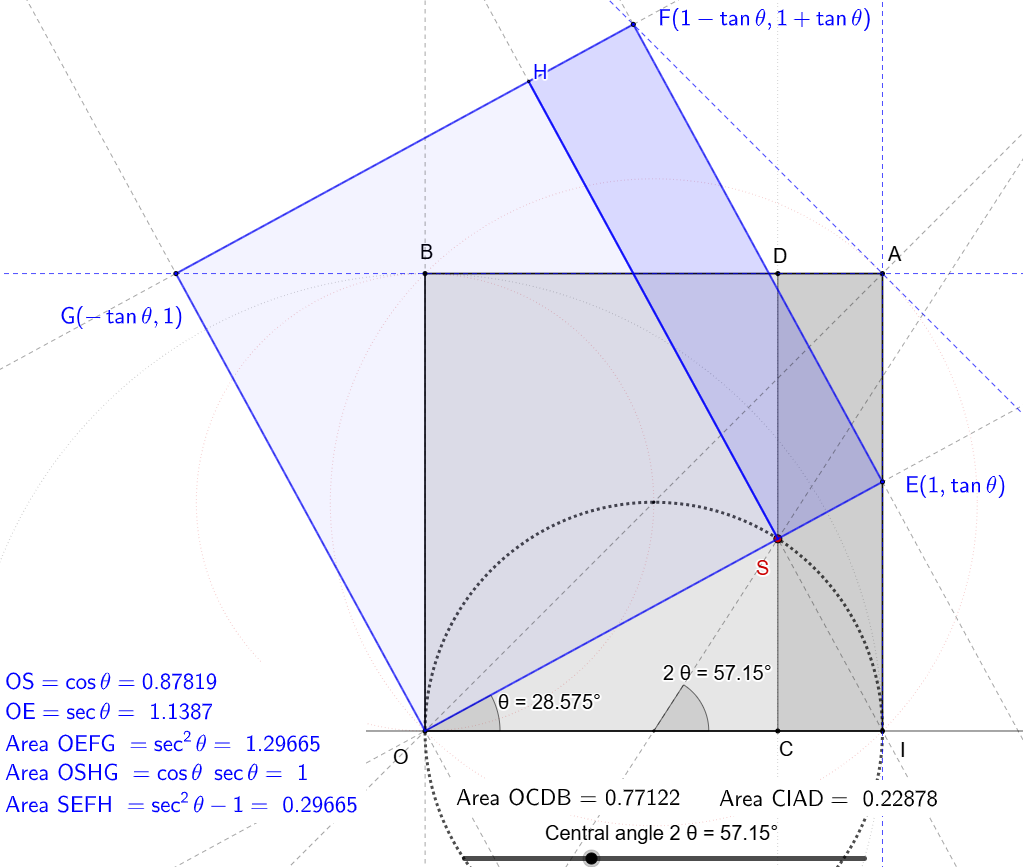}
        \caption{The expanding system framework.}
        \label{fig:unit_cell_b}
    \end{subfigure}
    
    \vspace{8pt} 
    
    \begin{subfigure}{0.75\textwidth} 
        \centering
        \includegraphics[width=\textwidth]{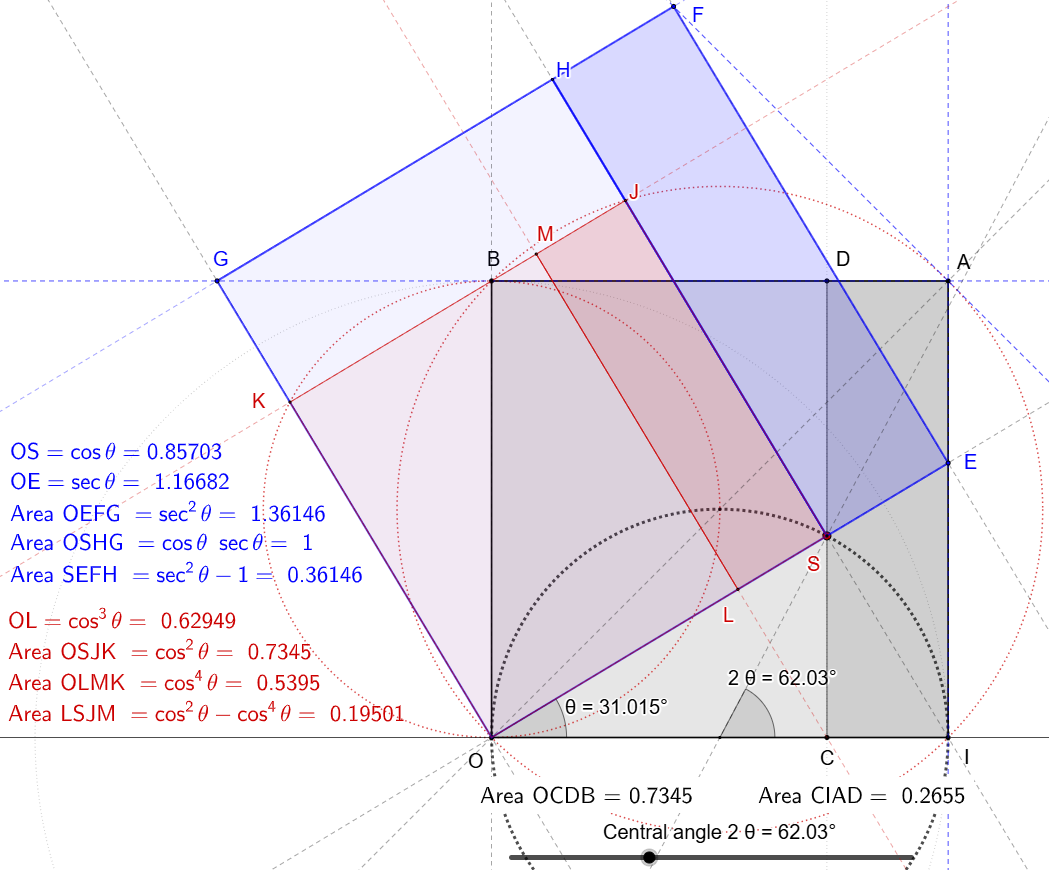}
        \caption{The complete superimposed system showing joint invariants.}
        \label{fig:unit_cell_c}
    \end{subfigure}

\caption{Step-by-step geometric construction of the foundational unit framework: 
    (a) the orthogonal partition of the unit reference square and its corresponding chord segments; 
    (b) the expanding configuration demonstrating the preservation of unit area via the cutting track $ISH$; 
    (c) the complete configuration demonstrating the reciprocal joint invariants. 
    Presented seed angles are randomly chosen. Interactive, live-variable versions of these constructions are available on GeoGebra, allowing the reader to continuously modify the angular parameter $\theta$: 
    see \href{https://www.geogebra.org/m/ymx8kdbp}{Figure (a)}, 
    \href{https://www.geogebra.org/m/urpagrem}{Figure (b)}, and 
    \href{https://www.geogebra.org/m/e7tc7aex}{Figure (c)}.}
    \label{fig:complete_unit_cell}
\end{figure}

\textbf{Exploration.} 
By continuously varying the parameter $\theta$ within an interactive dynamic geometry environment, the shifting configurations manifest clear geometric trajectories. As the point $S$ moves along its circular arc on $C_u$, the vertices $K$ and $P$ of the chord squares trace circular paths of diameter $OB$ and $IA$, respectively, while the vertices $J$ and $Q$ move along the circle of diameter $OA$. This configuration functions as a geometric root protractor, transforming the quadratic root segments of the diameter into a linear progression along the perimeter of the circle. This spatial arrangement reveals that both the external chord squares and the internal split rectangles share identical area-preservation behaviors.

\begin{proposition}[Pythagorean Partition Law]
For any seed angle $\theta$, the segments partitioning the base diameter satisfy $OC = \cos^2\theta$ and $CI = \sin^2\theta$. Consequently, the chord lengths are the geometric square roots of the segments partitioning the diameter ($OS = \cos\theta$ and $SI = \sin\theta$), and the internal split rectangles share identical area values with their corresponding external chord squares:
\begin{equation}
\text{Area}(OCDB) = \text{Area}(OSJK) = \cos^2\theta
\end{equation}
\begin{equation}
\text{Area}(CIAD) = \text{Area}(SIPQ) = \sin^2\theta
\end{equation}
\end{proposition}

\subsection{The Expanding System and Proportional Replication}

\textbf{Construction.} 
The configuration established in the preparatory stage allows for the systematic replication of proportions using similar triangles. As shown in Figure~\ref{fig:unit_cell_b}, the seed ray $OS$ is extended outward beyond the unit circle until it intersects the vertical line $x=1$ at the vertex $E(1, \tan\theta)$. Using the segment $OE$ as a base, we construct an expanding square $OEFG$. 

We project a line through points $I$ and $S$, extending it outward until it intersects the upper boundary of the square $OEFG$ at point $H$. Because the inscribed Thales angle is $\angle OSI = \frac{\pi}{2}$, the segment $IS$ is orthogonal to the base ray $OE$, which forces the line $ISH$ to run parallel to the lateral edges $OG$ and $EF$. This linear cutting track $ISH$ partitions the expanding square $OEFG$ into a left rectangular slice $OSHG$ and a remaining right rectangle $SEFH$.

\textbf{Exploration.} 
As the parameter $\theta$ is varied across its continuous range, the moving vertices of the expanding square trace distinct, regular geometric trajectories. The lower-right vertex $E$ slides along the vertical boundary line $x=1$, while the upper-left vertex $G$ is constrained to move along the horizontal line $y=1$ containing the top edge of the reference square. Consequently, the outer vertex $F$ traces a steady linear path along the diagonal track $y=2-x$. Throughout this variation, the internal partition line $ISH$ continuously maintains a structural copy of the reference unit proportions inside the shifting outer square.

\begin{proposition}[Expanding Invariants]
For any seed angle $\theta$ in the interval $(0, \pi/2)$, the dimensions and partitions of the expanding square $OEFG$ satisfy the following geometric invariants:
\begin{enumerate}
    \item The total side length is $OE = \sec\theta$, yielding a total scaling area of:
    \begin{equation}
    \text{Area}(OEFG) = \sec^2\theta
    \end{equation}
    \item The vertex coordinates are analytically given by $E(1, \tan\theta)$, $G(-\tan\theta, 1)$, and $F(1-\tan\theta, 1+\tan\theta)$.
    \item The left rectangular partition $OSHG$ remains perfectly invariant in area against changes to the seed angle:
    \begin{equation}
    \text{Area}(OSHG) = OS \times OG = \cos\theta \times \sec\theta = 1
    \end{equation}
\end{enumerate}
\begin{proof}
By the properties of right triangle similarities ($\triangle OSI \sim \triangle OIE$), since the adjacent side length is $OI=1$, the hypotenuse is exactly $OE = \sec\theta$. Because $OEFG$ is a square, the perpendicular side is $OG = \sec\theta$ at an angle of $\theta + \pi/2$ with the $x$-axis, yielding coordinates $(-\sec\theta \sin\theta, \sec\theta \cos\theta) = (-\tan\theta, 1)$. The width of the rectangular partition $OSHG$ is defined by the chord segment $OS = \cos\theta$, and its height matches the side length $OG = \sec\theta$. Their product yields a constant area of $1$ for all valid $\theta$.
\end{proof}
\end{proposition}

\subsection{The Contracting System and Joint Invariants}

\textbf{Construction.} 
To complete the reciprocal framework, we reintroduce the internal contracting square $OSJK$ ($\text{Area} = \cos^2\theta$) established in Section 2.1 and superimpose it with the external expanding square $OEFG$. As shown in Figure~\ref{fig:unit_cell_c}, the internal partitioning of this contracting square is driven by a direct geometric duality to the expanding system. A cutting line $LM$ is projected through the horizontal projection point $C$ running parallel to the chord segment $SI$. Because the inscribed Thales angle is $\angle OSI = \frac{\pi}{2}$, the segment $SI$ is orthogonal to the base side $OS$. Consequently, the parallel line through $C$ intersects the interior of the square along the segment $LM$, which is strictly perpendicular to $OS$ and parallel to the lateral edges $OK$ and $SJ$, dividing the square into two distinct rectangles: $OLMK$ and $LSJM$.

\textbf{Exploration.} 
As the parameter $\theta$ varies continuously, the moving vertices of the expanding square slide smoothly along their respective linear tracks ($x=1$, $y=2-x$, and $y=1$), while the vertices $J$ and $K$ of the contracting square trace circular paths within their respective domains. Despite the expanding square scaling outward toward infinity and the contracting square scaling inward toward the origin, the local area interactions obey a strict global reciprocity, visually grounding the continuous framework to the unit reference square $OIAB$.

\begin{proposition}[Reciprocal Joint Invariants]
For any parameter $\theta$ in the interval $(0, \pi/2)$, the dimensions and partitions of the joint system satisfy the following geometric invariants:
\begin{enumerate}
    \item The cutting line $LM$ intercepts the base side at a length of $OL = \cos^3\theta$.
    \item The internal split rectangles yield exact higher-order trigonometric areas:
    \begin{equation}
    \text{Area}(OLMK) = \cos^4\theta
    \end{equation}
    \begin{equation}
    \text{Area}(LSJM) = \cos^2\theta\sin^2\theta
    \end{equation}
    \item The cross-multiplied product of the expanding and contracting squares remains perfectly conserved against the area of the unit reference square $OIAB$:
    \begin{equation}
    \text{Area}(OSJK) \times \text{Area}(OEFG) = \cos^2\theta \times \sec^2\theta = 1 = \text{Area}(OIAB)
    \end{equation}
\end{enumerate}
\begin{proof}
Because the line $LM$ is constructed parallel to chord $SI$ through point $C$, the right triangle $\triangle OLC$ is similar to the foundational triangle $\triangle OSI$. This yields the base scaling relation $OL = OC \cos\theta$. Substituting $OC = \cos^2\theta$ gives the exact intercept length $OL = \cos^3\theta$. Since the total height of the square is $OS = \cos\theta$, the area of the left rectangle is $\text{Area}(OLMK) = OL \times OS = \cos^3\theta \times \cos\theta = \cos^4\theta$. The complementary area follows naturally as $\cos^2\theta - \cos^4\theta = \cos^2\theta(1-\cos^2\theta) = \cos^2\theta\sin^2\theta$. Finally, because $\sec\theta = \frac{1}{\cos\theta}$, the global area product simplifies identically to $1$ for all valid $\theta$.
\end{proof}
\end{proposition}

\section{Global Unfolding: Connecting Rotation to the Linear Power Map}

\textbf{Continuous Unfolding.}
By recursively cascading the reciprocal scaling logic established in Section 2, the localized partitions unfold into a continuous, multi-generational planar spiral (Figure~\ref{fig:global_spiral}). Rather than treating expansion and contraction as isolated transformations, the geometric operations iterate outward and inward recursively from the origin $O$. The angular parameter $\theta$ functions as the singular metric driving this progression, causing successive generations of squares $S_n$ to scale exponentially across the 2D plane. 

Throughout this continuous variation, the entire system behaves as a coherent geometric network. As the multi-generation squares expand under continuous angular rotation, their moving vertical and horizontal edge lines continuously slice across the fixed, linear coordinate axes of the plane. By tracking the exact coordinates where these moving boundaries intercept the axes, a direct correspondence is explicitly formalized. This mapping links our continuous angular parameter $\theta$ directly back to the sliding linear seed $s$ established in our linear Power Map~\cite{dijksman2026powermap}, mapping a purely rotational mechanism onto a linear field of rational exponents.

\begin{figure}[t!]
    \centering
    \includegraphics[width=0.9\textwidth]{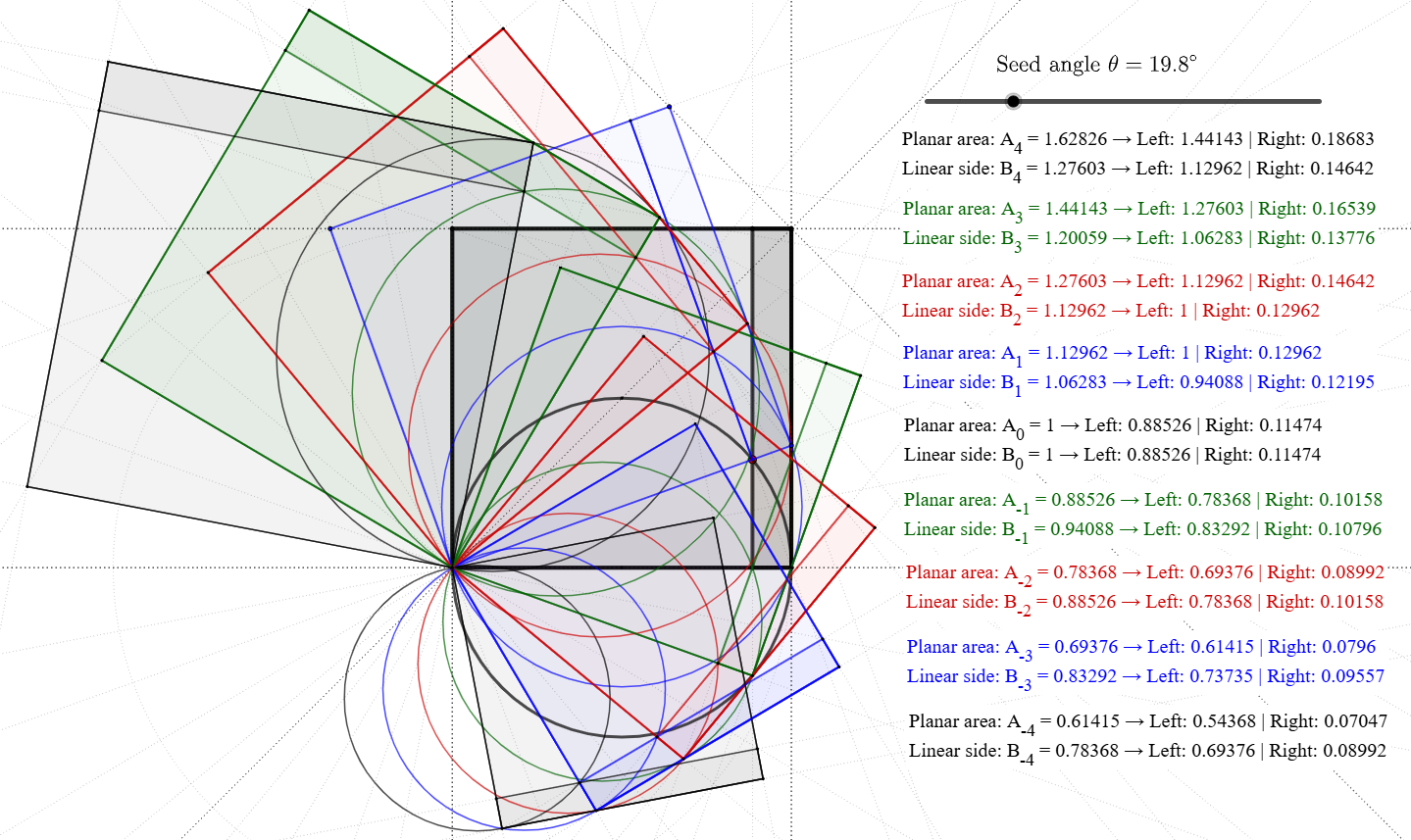} 
    \caption{Continuous angular variation of the multi-generational Power Spiral Map. Modulation of the parameter $\theta$ generates a recursive scaling lattice. Presented seed angle is randomly chosen. Interactive animation available at the \href{https://www.geogebra.org/m/bgk59tua}{GeoGebra Workspace}.}
    \label{fig:global_spiral}
\end{figure}

\textbf{Exploration.}
An interactive exploration reveals that the proportional replication of partitions provides an explicit spatial visualization of polynomial terms assembling along a single geometric sequence. By actively modulating the parameter $\theta$ within the dynamic construction, the reader can observe how the internal cutting lines physically aggregate individual algebraic powers across the coordinate field, as documented sequentially in Table~\ref{tab:polynomial_partitions}.

For a polynomial base defined by $x = \sec\theta$, the generations display perfect structural uniformity. In the first expansion generation $S_{1}$, the total area scales to $A_{1} = x^2$, leaving an invariant unit area block $A_{1L} = 1$ as its left partition, while its total base length scales to $B_{1} = x$ with a left linear partition segment $B_{1L} = x^{-1}$. Moving outward to $S_{2}$, the left rectangular cut scales precisely to the quadratic area variant $A_{2L}=x^2$ (leaving a right partition of $A_{2R}=x^4 - x^2$), while its base side partitions into segments of $B_{2L}=1$ and $B_{2R}=x^2 - 1$. This proportional behavior extends symmetrically down through the contracting domain ($n<0$), providing a clean geometric representation of successive powers through purely planar intersections.

\begin{proposition}[Global Scaling and Partition Laws]
Let the multi-generational scaling family of squares $S_n$ be parameterized by the base $x = \sec\theta$ for $n \in \mathbb{Z}$. The total dimensions and internal linear and planar partitions of any generation $S_n$ conform to a strict geometric sequence governed by the following operational invariants:
\begin{equation}
B_n = x^n, \quad B_{nL} = x^{n-2}, \quad B_{nR} = x^{n-2}(x^2 - 1)
\end{equation}
\begin{equation}
A_n = x^{2n}, \quad A_{nL} = x^{2n-2}, \quad A_{nR} = x^{2n-2}(x^2 - 1)
\end{equation}
\end{proposition}
\begin{proof}
By construction, each successive generation $S_n$ is mapped via a recursive scaling operator that increments the total side length by a linear factor of $x = \sec\theta$, yielding $B_n = x^n B_0$. Since the unit reference square $S_0$ possesses a baseline side length of $B_0 = 1$, it follows that $B_n = x^n$, and its total planar area scales quadratically as $A_n = B_n^2 = x^{2n}$. 

Preserving a self-similar structural partition across the global network requires the linear partition ratio to remain constant across all generations, such that $B_{nL}/B_n = x^{-2} = \cos^2\theta$. This fixes the left linear boundary segment at $B_{nL} = x^{n-2}$ for any integer generation $n$. Because the left internal partition $A_{nL}$ forms a rectangular domain spanning the full height of its host square ($B_n$) across a width of $B_{nL}$, its planar area is given by the cross-product:
\begin{equation}
A_{nL} = B_n \times B_{nL} = x^n \times x^{n-2} = x^{2n-2}
\end{equation}
The complementary right linear and planar partitions $B_{nR}$ and $A_{nR}$ follow naturally by direct algebraic subtraction, preserving a perfectly uniform progression of even-power area partitions across the entire infinite domain of the spiral.
\end{proof}

\begin{table}[H]
    \centering
    \renewcommand{\arraystretch}{1.4}
    \setlength{\tabcolsep}{4pt}
    \begin{tabular}{l c c c c c c}
        \toprule
        & \multicolumn{3}{c}{\textbf{Planar Area ($A$)}} & \multicolumn{3}{c}{\textbf{Linear Base ($B$)}} \\
        \cmidrule(r){2-4} \cmidrule(l){5-7}
        \shortstack{\textbf{Gen.} \\ ($S_n$)} & 
        \shortstack{\textbf{Total} \\ \textbf{Area ($A_n$)}} & 
        \shortstack{\textbf{Left} \\ \textbf{Part ($A_{nL}$)}} & 
        \shortstack{\textbf{Right} \\ \textbf{Part ($A_{nR}$)}} & 
        \shortstack{\textbf{Total} \\ \textbf{Side ($B_n$)}} & 
        \shortstack{\textbf{Left} \\ \textbf{Part ($B_{nL}$)}} & 
        \shortstack{\textbf{Right} \\ \textbf{Part ($B_{nR}$)}} \\
        \midrule
        $S_{4}$  & $x^8$ & $x^6$ & $x^8 - x^6$ & $x^4$ & $x^2$ & $x^4 - x^2$ \\
        $S_{3}$  & $x^6$ & $x^4$ & $x^6 - x^4$ & $x^3$ & $x$   & $x^3 - x$   \\
        $S_{2}$  & $x^4$ & $x^2$ & $x^4 - x^2$ & $x^2$ & $1$   & $x^2 - 1$   \\
        $S_{1}$  & $x^2$ & $1$   & $x^2 - 1$   & $x$   & $x^{-1}$ & $x - x^{-1}$ \\
        $S_0$    & $1$   & $x^{-2}$ & $1 - x^{-2}$ & $1$   & $x^{-2}$ & $1 - x^{-2}$ \\
        $S_{-1}$ & $x^{-2}$ & $x^{-4}$ & $x^{-2} - x^{-4}$ & $x^{-1}$ & $x^{-3}$ & $x^{-1} - x^{-3}$ \\
        $S_{-2}$ & $x^{-4}$ & $x^{-6}$ & $x^{-4} - x^{-6}$ & $x^{-2}$ & $x^{-4}$ & $x^{-2} - x^{-4}$ \\
        $S_{-3}$ & $x^{-6}$ & $x^{-8}$ & $x^{-6} - x^{-8}$ & $x^{-3}$ & $x^{-5}$ & $x^{-3} - x^{-5}$ \\
        $S_{-4}$ & $x^{-8}$ & $x^{-10}$ & $x^{-8} - x^{-10}$ & $x^{-4}$ & $x^{-6}$ & $x^{-4} - x^{-6}$ \\
        \bottomrule
    \end{tabular}
    \caption{Algebraic distribution and structural partitions across a 9-generation scaling family parameterized by the base $x = \sec\theta$. Index values $n > 0$ denote expanding outer configurations, $n < 0$ denote contracting inner configurations, and $S_0$ isolates the unit reference square. This structural alignment illustrates the simultaneous tracking of higher-order 2D area partitions and 1D linear boundaries via continuous rotation.}
    \label{tab:polynomial_partitions}
\end{table}

\section{Geometric Realization: Deriving the Golden Ratio and Plastic Ratio}

While classical static spirals, such as the Spiral of Theodorus or the logarithmic \textit{Spira Mirabilis}, primarily trace recursive progressions of linear lengths, the Power Spiral Map correlates 1D side lengths with 2D area partitions (as detailed in Table~\ref{tab:polynomial_partitions}). Continuous variation of the angular parameter $\theta$ systematically varies the area and length ratios across co-evolving generations. 

The system's geometry is defined by the association of two distinct scaling regimes originating at the origin $O$. While the contracting system ($n < 0$) yields successive generations $S_{-m}$ (for $m = -n$) with total side lengths defined by $B_{-m} = \cos^m\theta$, the alignment is physically established via the expanding system. This configuration is achieved constructively using a circular tracking arc centered at $O$ with a radius defined by the expanding length $R = \sec^m\theta$. The compass sweeps this expanding metric across the plane to compare it directly against the framework's primary orthogonal boundaries, specifically the base ray of the generation $1$ square and the upper edge of the generation $-1$ square.

An alignment configuration is isolated when this tracking arc structurally coincides with these specific boundaries, producing the remarkable simultaneous coordinate intersections described in the Figure~\ref{fig:geometric_alignments} legends. Because the expanding and contracting domains are strictly proportional inverses, this physical coincidence of the expanding arc mathematically forces the corresponding contracting side length $B_{-m}$ to precisely equal the vertical projection length of the seed ray ($\tan\theta$). This fundamental geometric balance yields the tracking condition:
\begin{equation}
\cos^m\theta = \tan\theta = \frac{\sin\theta}{\cos\theta} \implies \cos^{m+1}\theta = \sin\theta
\end{equation}

Squaring both sides expresses the relationship entirely in terms of the fundamental circular functions:
\begin{equation}
\cos^{2m+2}\theta = \sin^2\theta = 1 - \cos^2\theta
\end{equation}

Let $A_1 = x^2 = \sec^2\theta$ represent the area of the first expanding square $S_1$, meaning its reciprocal value maps to the shrinking area metric $\frac{1}{A_1} = \cos^2\theta$. Substituting this metric directly into our angular relation yields $\left(\frac{1}{A_1}\right)^{m+1} + \frac{1}{A_1} - 1 = 0$. Multiplying the entire equation by $-A_1^{m+1}$ yields a clean family of polynomial identities mapping to specific alignment configurations:
\begin{equation}
A_1^{m+1} - A_1^m - 1 = 0
\end{equation}

\renewcommand{\textfraction}{0.05}
\renewcommand{\topfraction}{0.95}
\renewcommand{\floatpagefraction}{0.75}

\begin{figure}[!htbp]
    \centering
    \begin{subfigure}{\textwidth}
        \centering
        \includegraphics[width=0.55\textwidth]{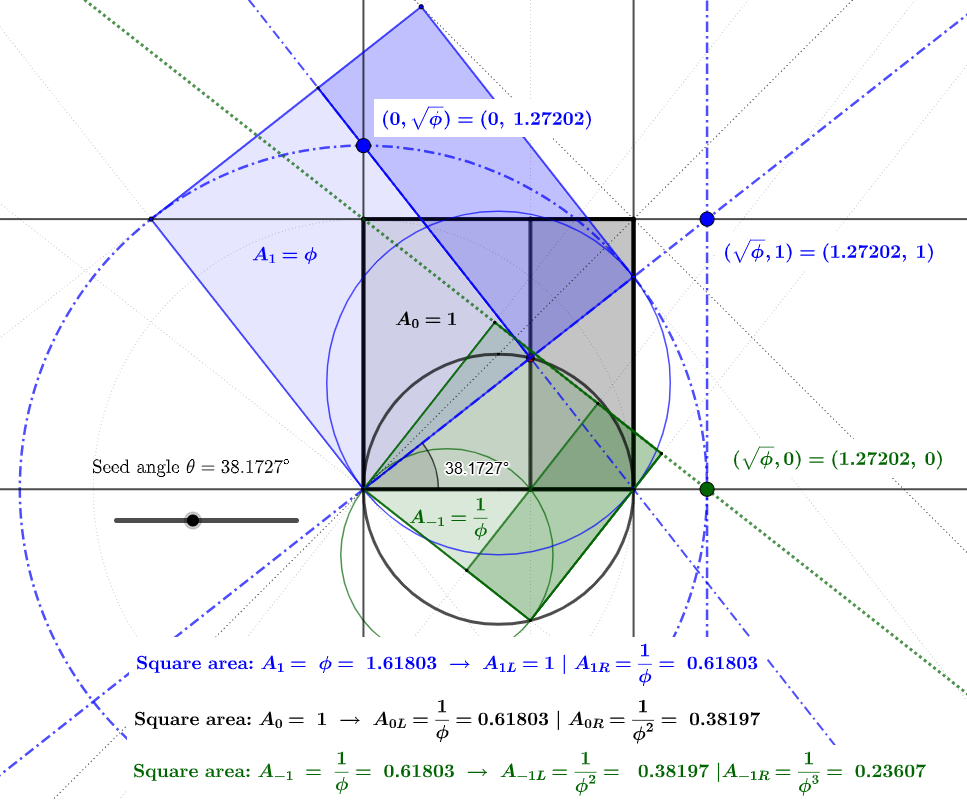}
        \caption{Golden Ratio alignment ($m=1$).}
        \label{fig:align_golden}
    \end{subfigure}

    \nopagebreak\vspace{0.5cm}

    \begin{subfigure}{\textwidth}
        \centering
        \includegraphics[width=0.95\textwidth]{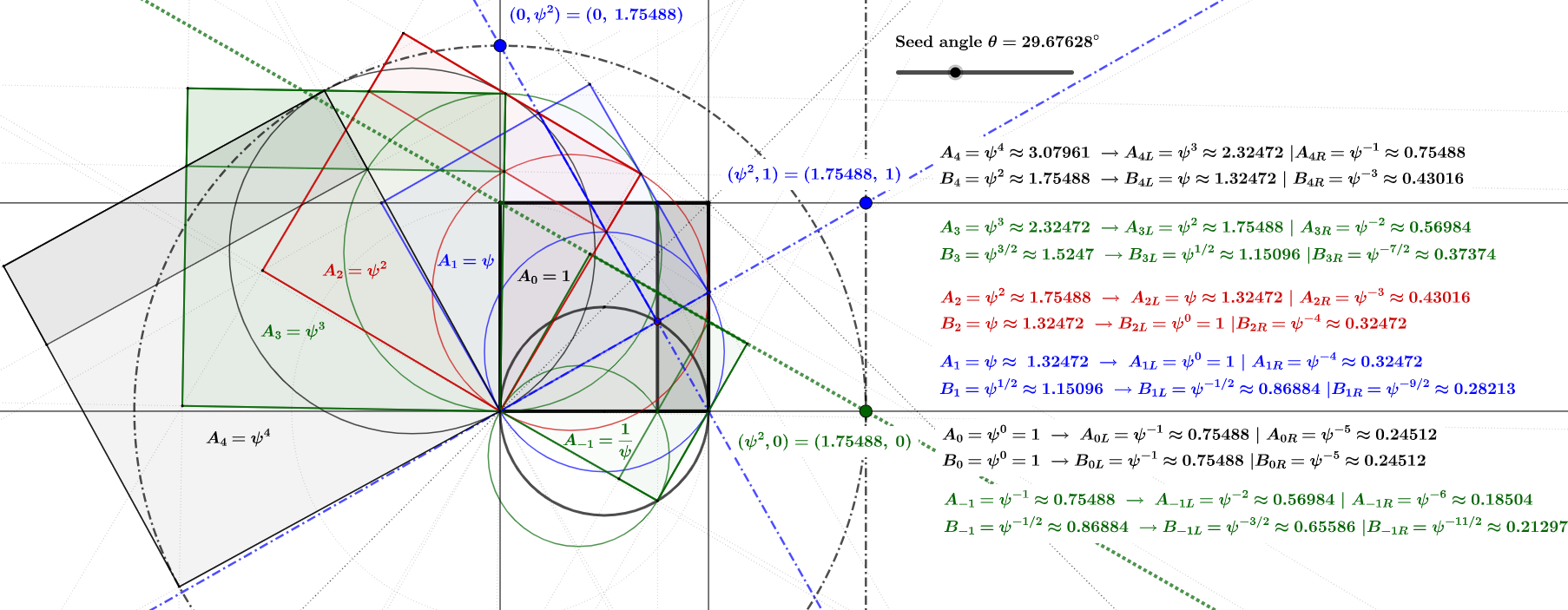}
        \caption{Plastic Ratio alignment ($m=4$).}
        \label{fig:align_plastic}
    \end{subfigure}
    \caption{Geometric realization of alignments. (a) Seed angle $\theta \approx 38.1727^\circ$ corresponds to the golden ratio alignment $\phi$ via three boundary intersections; see \href{https://www.geogebra.org/m/us3nd5zf}{Figure (a) and GeoGebra construction}. (b) Angle $\theta \approx 29.6763^\circ$ isolates a configuration that solves the polynomial for the plastic ratio $A_1 = \psi \approx 1.32472$ (where the second-generation area $A_2 = A_1^2 = \psi^2 \approx 1.75488$); see \href{https://www.geogebra.org/m/dqeqfgmg}{Figure (b) and GeoGebra construction}.}
    \label{fig:geometric_alignments}
\end{figure}

This identity reveals specific algebraic constants, mapped explicitly across the two distinct states of our dual-panel visualization:

\begin{itemize}
    \item \textbf{The Golden Ratio ($\phi$) [$m=1$]:} At the primary alignment configuration ($\theta \approx 38.17^\circ$), the equation resolves to the quadratic $A_1^2 - A_1 - 1 = 0$. As shown in Figure~\ref{fig:align_golden}, this state produces three simultaneous coordinate intersections along the boundaries: the horizontal base extension at $(\sqrt{\phi}, 0)$, the vertical projection coincidence at $(\sqrt{\phi}, 1)$, and the vertical axis intercept at $(0, \sqrt{\phi})$.
    \item \textbf{The Plastic Ratio ($\psi$) [$m=4$]:} At the higher-order alignment configuration, the required seed angle is $\theta \approx 29.6763^\circ$. Here, the system satisfies the fifth-order polynomial $A_1^5 - A_1^4 - 1 = 0$, which factors into $(A_1^3 - A_1 - 1)(A_1^2 + A_1 + 1) = 0$. The real component yields the cubic relation $A_1^3 - A_1 - 1 = 0$, isolating the Plastic Ratio ($\psi \approx 1.32472$) governing the Padovan and Perrin series. As shown in Figure~\ref{fig:align_plastic}, this reveals a sequential bridge where unit blocks of area 1 occur between expanding exponent tracks. This configuration produces a distinct triad of coordinate coincidences on the grid lines at the square of the ratio: the horizontal extension $(\psi^2, 0)$, the vertical projection coincidence $(\psi^2, 1)$, and the vertical axis intercept $(0, \psi^2)$.
\end{itemize}

This construction suggests a broader family of geometric alignments associated with the polynomial $A^{m+1}-A^{m}-1=0$, demonstrating that certain intersection configurations correspond directly to the roots of polynomial identities via the continuous rotation of a singular, planar parameter.

\section{Discussion and Historical Precedents}

The Power Spiral Map relates to a history of geometric construction that extends from 14th-century modular partitions to early modern continuous instrumentation. This recursive scaling lattice shares structural themes with historical geometric progressions, such as those preserved in the 16th-century manuscript compendium \textit{Persan 169}~\cite{BnFPersan169}. Originally held in the collection of the 17th-century polymath Melchisédech Thévenot and drawing upon the 14th-century scaling methods of Ab\={u} Bakr al-\d{H}al\={\i}l, this artifact illustrates a discrete logarithmic spiral formed by connected right triangles (Figure~\ref{fig:historical_precedents}a). 

While the classical Spiral of Theodorus grows arithmetically with side lengths scaling as $\sqrt{n}$, the \textit{Persan 169} construction utilizes a strict geometric progression where the values inside the radical scale by a constant factor $k = \frac{4}{3}$:
\begin{equation}
\dots \rightarrow \sqrt{72} \rightarrow \sqrt{96} \rightarrow \sqrt{128} \rightarrow \sqrt{\frac{512}{3}} \rightarrow \dots
\end{equation}
Consequently, successive spatial vector lengths scale exponentially by $\sqrt{k}$, mapping the power of the radius directly to the angular step. By transitioning our framework from a parametric linear lattice to a continuous spiral, this model situates itself within the long-term evolution of these early geometric transformations, bridging the discrete scaling of early treatises with the mechanical linkage systems of Descartes and Van Schooten~\cite{vanschooten1646} and the projective coordinate mappings of d'Ocagne's nomography~\cite{ocagne1899}.

Similar recursive partition structures appear in the independent geometric models of Nguyen Tan Tai~\cite{nguyen2012structure}. His self-similar partition cascades driven by an angular parameter (reproduced in Figure~\ref{fig:historical_precedents}b) explored how localized area-preserving partitions can propagate outward to form recursive, space-filling scaling networks.

\begin{figure}[t!]
    \centering
    \begin{subfigure}{0.48\textwidth}
        \centering
        \includegraphics[width=\textwidth]{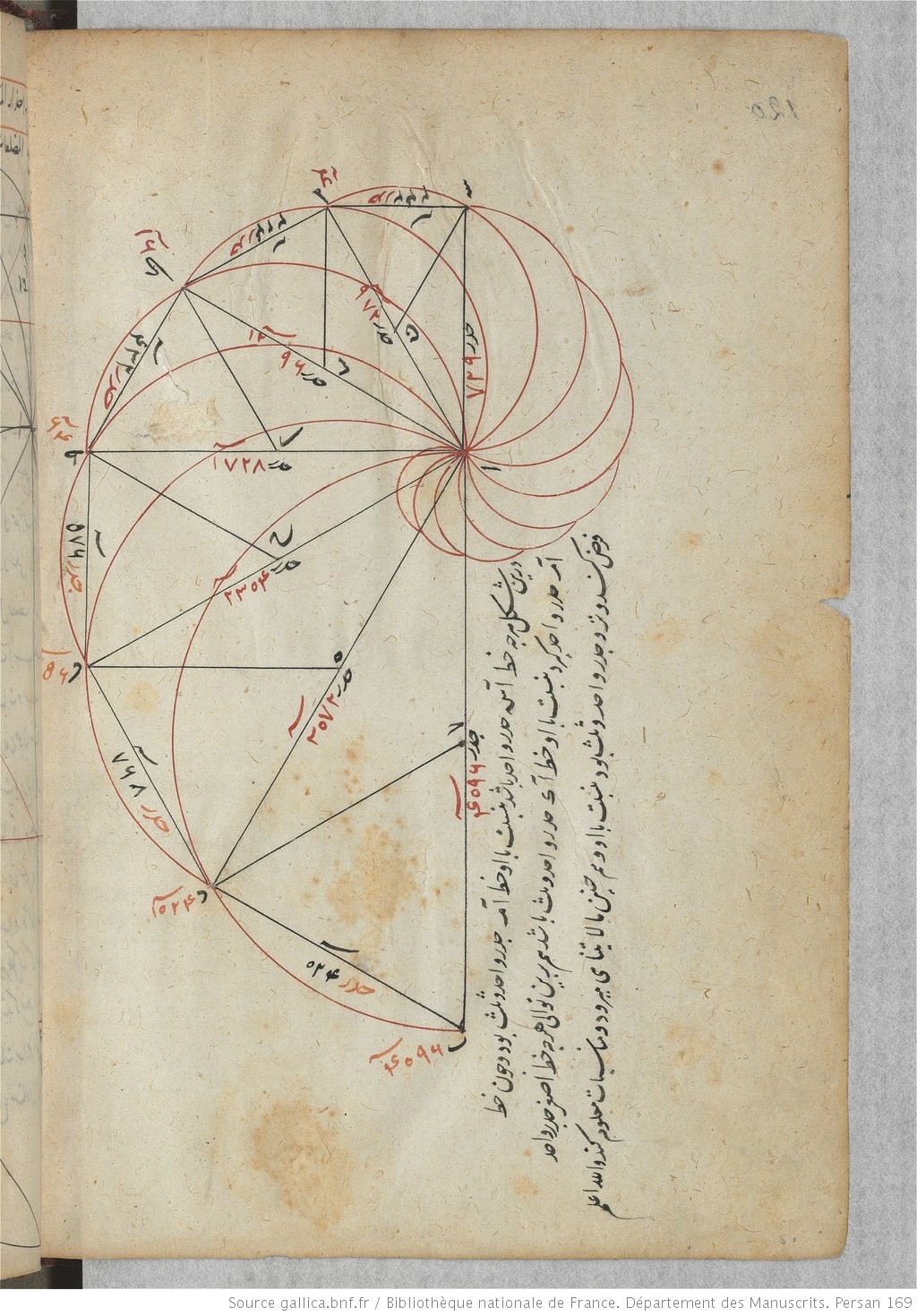}
        \caption{Persian logarithmic radicand spiral (BnF Persan 169).}
        \label{fig:persian_spiral}
    \end{subfigure}
    \hfill
    \begin{subfigure}{0.48\textwidth}
        \centering
        \includegraphics[width=\textwidth]{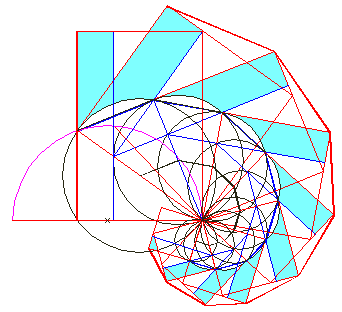}
        \caption{Nguyen Tan Tai's Fig-8-1 framework.}
        \label{fig:nguyen_framework}
    \end{subfigure}
    \caption{Historical precursors to the geometric locus framework. (a) Folio 120v of the 16th-century Persian manuscript illustrating a recursive geometric progression under the radical with a constant scaling factor of $\frac{4}{3}$ (Source: Bibliothèque nationale de France). (b) The spatial element spiral from the digital archives of Nguyen Tan Tai. Driven by the angle $\alpha$, the construction illustrates a self-similar partition cascade across circles, squares, and rectangles~\cite{nguyen2012structure}.}
    \label{fig:historical_precedents}
\end{figure}

By linking these discrete, area-preserving invariants with continuous angular rotation, the Power Spiral Map complements this structural lineage. It demonstrates that certain intersection configurations correspond directly to the roots of polynomial identities within a planar framework, providing an intuitive visual link connecting classical synthetic geometry with algebraic identities.

\bibliographystyle{unsrt} 
\bibliography{references}

\end{document}